\documentclass[11pt]{amsart}
\usepackage{amssymb,amsthm,amsfonts}
\newtheorem{theorem}{Theorem}[section]

\newtheorem{lemma}[theorem]{Lemma}
\newtheorem{corollary}[theorem]{Corollary}
\newtheorem{proposition}[theorem]{Proposition}
\newtheorem{definition}[theorem]{Definition}

\newtheorem{remark}[theorem]{Remark}

\newcommand{\field}[1]{\mathbb{#1}}

\newcommand{\N}{\field{N}}

\newcommand{\Q}{\field{Q}}
\newcommand{\R}{\field{R}}

\newcommand{\Cal}[1]{\mathcal{#1}}
\newcommand{\bl}{{\bf l}}
\newcommand{\bm}{{\bf m}}
\newcommand{\bo}{{\bf 0}}
\newcommand{\bp}{{\bf p}}

\newcommand{\bS}{{\bf S}}

\newcommand{\cC}{\Cal{C}}
\newcommand{\cE}{\Cal{E}}

\newcommand{\cI}{\Cal{I}}
\newcommand{\cK}{\Cal{K}}

\newcommand{\cO}{\Cal{O}}

\newcommand{\clos}{{\rm clos}}

\newcommand{\dist}{{\rm dist}}

\newcommand{\gm}{{\gamma}}
\newcommand{\Gm}{{\Gamma}}

\newcommand{\omg}{\omega}

\newcommand{\Qbar}{\overline{\Q}}
\newcommand{\oo}{\underline{0}}

\newcommand{\sgm}{\sigma}

\newcommand{\ud}{{\rm d}}
\newcommand{\ula}{{\underline{a}}}
\newcommand{\ulb}{{\underline{b}}}

\newcommand{\wtE}{\widetilde{E}}

\newcommand{\bs}{ {\tiny $\blacksquare$} \\}
\numberwithin{equation}{section}

\numberwithin{equation}{section}

\begin{document}
\title[Blow-Analytic Equivalence versus $\cK$-bi-Lipschitz Equivalence]{Blow-Analytic Equivalence versus Contact bi-Lipschitz Equivalence}
\author[L. Birbrair]{Lev Birbrair}\thanks{L.~Birbrair was partially supported by CNPq-Brazil grant 300575/2010-6}
\author[A. Fernandes]{Alexandre Fernandes}\thanks{A.~Fernandes was partially supported by CNPq-Brazil grant 302998/2011-0}
\author[T. Gaffney]{Terence Gaffney}\thanks{T.~Gaffney was partially supported by PVE-CNPq Proc. 401565/2014-9}
\author[V. Grandjean]{Vincent Grandjean}
\address{L.~Birbrair, A.~Fernandes \& V.~Grandjean - 
Departamento de Matem\'atica, Universidade Federal do Cear\'a Av.
Humberto Monte, s/n Campus do Pici - Bloco 914, 60455-760
Fortaleza-CE, Brazil} 
\email{birb@ufc.br}
\email{alex@mat.ufc.br}
\email{vgrandjean@mat.ufc.br}
\address{T.~Gaffney,  Mathematics department, Northeastern University, 360 Huntington Ave., 02115 Boston-MA, USA}
\email{t.gaffney@neu.edu}
%
%
%
%
%
%
\date{\today}
%
%
\begin{abstract}
{The main result of this note is that two blow-analytically equivalent 
real analytic plane function germs are sub-analytically bi-Lipschitz contact equivalent.}
\end{abstract}
\maketitle


\section{Introduction}\label{section:intro}
The bi-Lipschitz classification of regular (analytic, smooth,  definable) function germs is rather recent.
Beyond the plane case, very little is known.
As Parusi\'nski and Henry showed in \cite{HePa}, the bi-Lipschitz right equivalence of real analytic plane function
germs, unlike the classical case of the topological right equivalence \cite{Fuk}, already presents moduli. 

The blow-analytic equivalence (of real analytic function germs) introduced more than thirty years ago by 
Kuo \cite{Ku1}, nevertheless has no moduli. Roughly speaking a blow-analytic 
homeomorphism is obtained by blowing-down an analytic isomorphism between blown-up manifolds (obtained 
by finite composition of blowings-up at the source and at the target respectively). Although not bi-Lipschitz in general it is 
still a rigid homeomorphism.

Birbrair, Fernandes \& Grandjean, jointly with A. Gabrielov, recently exhibited a complete invariant, 
called \em minimal pizza, \em of the sub-analytic bi-Lipschitz contact equivalence of Lipschitz sub-analytic plane 
function germs \cite{BFGG}.  The existence of this complete invariant implies that this equivalence 
has no moduli (result already known in for smooth plane germs from a previous work of Birbrair and Fernandes, 
jointly with Costa and Ruas \cite{BCFR}). 
A pizza is a way to encode, by means of finitely many rational numbers, all the sub-analytic asymptotic 
behaviors of the considered function at the considered point. 
As the result of \cite{BFGG} states such combinatorial data encoding the asymptotics of a given germ
is indeed of metric nature. The local monomialization/resolution of (sub-)analytic 
function germs, although using a parameterization (the associated blowing-down mapping),
is very convenient to investigate the sub-analytic asymptotic behaviors at the considered point. 
In the category of real analytic plane function germs, such hints suggest looking for relations 
between blow-analytic equivalence and bi-Lipschitz contact equivalence. 
Nevertheless Koike \& Parusi\'nski have already provided examples of 
bi-Lipschitz right-equivalent real analytic plane function germs (thus bi-Lipschitz contact equivalent) 
which are not blow-analytic equivalent \cite{KoPa2}. 

This note establishes the relation between the blow-analytic equivalence and the sub-analytical
bi-Lipschitz contact equivalence, namely 

\medskip\noindent
{\bf Theorem \ref{thm:main}.} \em
Blow-analytic equivalent real analytic plane function germs are sub-analytically bi-Lipschitz
contact equivalent.
\em 

\smallskip
The proof of this result is a consequence of the combinatorial local data of both blowings-up 
mapping in the definition of the (cascade) blow-analytic homeomorphism and of the combinatorial data of the 
blow-analytic equivalent function germs (once resolved). 
The analytic isomorphism inducing the homeomorphism preserves the combinatorial data of the pair 
"resolved function and corresponding resolution mapping" (one of the composition of blowings-up) and 
the other pair "resolved function and corresponding resolution mapping". 
The combinatorial data of a given pair function/resolution can be computed explicitly, 
using the parameterization in the resolved manifold, and this data can be used 
to produce all the ingredients that are needed to cook-up a pizza of the function. 
These ingredients encoded by finitely many rational numbers, are preserved by the blow-analytic homeomorphism,
because of the isomorphism between the resolved manifolds preserving the combinatorial data of the 
pair function/resolution. 
More precisely our key argument is that \em a blow-analytic isomorphism preserves 
contact between any two real-analytic half-branches as well as the normalized order of a real analytic
function along a real-analytic half-branch \em (see also \cite{KoPa1}). When combined with 
an equivalent criterion of sub-analytic bi-Lipschitz contact equivalence,
stated in Proposition \ref{prop:criterion}, the blow-analytic homeomorphism
maps a pizza (of one function) into an equivalent pizza (of the other function), so that 
Theorem \ref{thm:main} is true by the results of \cite{BFGG}. 

\bigskip
The paper is organized as follows.

Section \ref{section:BAE} recalls quickly the notion of blow-analytic equivalence and 
two properties of metric nature. Section \ref{section:bLCE} presents the (sub-analytic) 
bi-Lipschitz contact equivalence with a new equivalence criterion, Proposition \ref{prop:criterion}. 
It is followed by Section \ref{section:HP} where 
Lemma \ref{lem:res-1} shows the local normal form of any finite sequence of points blowings-up
at any point of the exceptional divisor. This constitutes what we call the Hsiang \& Pati (local) data 
of the blowing-down mapping, and further in this section we see how it 
reflects in any (cascade) blow-analytic homeomorphism. 
 Section \ref{section:main} provides a complete proof of the main result via 
several simple intermediary results leading to Proposition \ref{prop:BA-pizza} 
stating that the considered blow-analytic homeomorphism induces an equivalence of Pizza,
although stricto sensu not being one.
The last section proposes a sketch of a 
proof of the existence of pizzas for plane real analytic function germs, providing clues why the 
result presented here was to be expected.

%
%
%
%
%
%
%
%
%
%
%
%
%
%
%
%
%
%
%
%
%
%
%
%
%
%
%
%
%
%
%
\section{Blow analytic equivalence}\label{section:BAE}
We will only work with real analytic plane function germs, following the exposition
of \cite{KoPa1}. 

\medskip
We start by fixing some notation we will use in the whole paper.

\smallskip
Let $\cO_2$ be the local $\R$-algebra of real analytic function germs $(\R^2,\oo) \to \R$ at 
the origin $\oo$ of $\R^2$. Let $\bm_2$ be its maximal ideal.

\smallskip
Let $S$ be a real analytic regular surface with structural sheaf $\cO_S$. 
Let $\ula$ be a point of $S$. We denote by $\cO_{S,\ula}$ 
the $\R$-algebra of real analytic function germs $(S,\ula) \to \R$. If $I$ is any (coherent) 
$\cO_S$-ideal sheaf, let $I_\ula$ be the $\cO_{S,\ula}$-ideal induced at $\ula$.

\medskip
\begin{definition}\label{def:BAhomeo}
A homeomorphism germ $h:(\R^2,\oo) \to (\R^2,\oo)$ is \em blow-analytic \em if there exists a commutative diagram
\[
\begin{array}{lll}
(M,E) & \stackrel{\pi}{\longrightarrow} & (\R^2,\oo) \\
\,\,\,\,\, \Phi \, \downarrow &  & \;\; h \downarrow \\
(M',E') & \stackrel{\pi'}{\longrightarrow} & (\R^2,\oo) \\
\end{array}
\]
with the following properties
\\
- $\Phi$ is a real analytic isomorphism which induces $h$.
\\
- Both mappings $\pi$ and $\pi'$ are finite composition of points blowings-up and 
$E$ and $E'$ are (simple) normal crossing divisors.
\end{definition}
Consequently any blow-analytic homeomorphism is sub-analytic. 
Such homeomorphisms clearly form of sub-group of the homeomorphism germs $(\R^2,\oo) \to (\R^2,\oo)$. 

\medskip
The blow-analytic equivalence of two function germs of $\cO_2$ is thus defined as 
expected:
\begin{definition}\label{def:BAE}
Two real analytic  function germs $f,g:(\R^2,\oo) \to \R$ are \em blow-analytic equivalent \em if there 
exists a blow-analytic homeomorphism $h:(\R^2,\oo) \to (\R^2,\oo)$  such that 
$f = g\circ h$.
\end{definition}
The blow-analytic equivalence is a right-equivalence. A remarkable fact of this equivalence
is that it does not admit moduli \cite{Ku1}.

\medskip
An \`a-priori refinement of the notion of blow-analytic homeomorphism is found in the next
\begin{definition}\label{def:casBAE}
A homeomorphism germ $h:(\R^2,\oo) \to (\R^2,\oo)$ is \em cascade blow-analytic \em if there exists a commutative diagram
\[
\begin{array}{lllllll}
(M_k,E_k) & \stackrel{\pi_k}{\longrightarrow}
 & (M_{k-1},E_{k-1}) & \stackrel{\pi_{k-1}}{\longrightarrow} \cdots & \stackrel{\pi_2}{\longrightarrow} & (M_1,E_1) & 
\stackrel{\pi_0}{\rightarrow}(\R^2,\bo) \\
\,\,\,\,\Phi \, \downarrow &  & \,\,\,\, \, h_{k-1} \, \downarrow  &  & & \,\,\,\, h_1 \downarrow& \,\,\,\,\,\, h \downarrow \\
(M_k',E_k') & \stackrel{\pi_k'}{\longrightarrow}
 & (M_{k-1}',E_{k-1}') & \stackrel{\pi_{k-1}'}{\longrightarrow} \cdots & \stackrel{\pi_2'}{\longrightarrow} & (M_1',E_1') & 
\stackrel{\pi_0'}{\rightarrow}(\R^2,\bo) \\
\end{array}
\]
with the following properties
\\
(i) $\Phi$ is a real analytic isomorphism which induces $h$.
\\
(ii)Each $\pi_i$ (resp. $\pi_j'$) is the blowing-up of a point in $M_{i-1}$ (resp. $M_{i-1}'$).
\\
(iii) Each $h_i:M_i\to M_i'$ is a homeomorphism such that $h_i(E_i) = E_i'$. 
\end{definition}
The condition $(iii)$ implies that the center of the next blowing-up in $M_i$ must be mapped by 
$h_i$ onto the center of the next blowing-up in $M_i'$. 

\smallskip
Although a cascade blow-analytic homeomorphism germ seems to be far more rigid than just a 
blow-analytic homeomorphism germ, it is not so as states the next
\begin{theorem}[\cite{KoPa1}]\label{thm:CBA=BA}
A homeomorphism germ $h:(\R^2,\oo) \to (\R^2,\oo)$ is blow-analytic if and 
only if it is cascade blow-analytic.
\end{theorem}
As we will see in the proof of our main result in Section \ref{section:main}, 
the property described in Theorem \ref{thm:CBA=BA} is absolutely fundamental, since the combinatorial 
data controlling the blow-analytic maps and the equivalent functions (once resolved)
are depending mostly on the component(s) of the exceptional divisor(s).

\medskip
%
A (real analytic) \em half-branch \em is the image of the
restriction, to the germ at $0$ of the non-negative real half-line $\R_{\geq 0}$, of a non constant 
real analytic map germ $(\R,0) \to (\R^2,\oo)$. 

\smallskip
Let $\cC_1$ and $\cC_2$ be two half-branches. For $r$ positive and small enough,
let $\bS_r$ be the Euclidean sphere centered at $\oo$ of radius $r$. 
Let $c(r) = \dist(\cC_1\cap\bS_r, \cC_2\cap \bS_r)$. The function $c$ is sub-analytic
and thus admits a Puiseux expansion of the form 
\begin{center}
$c(r) = r^\beta\cdot (\alpha + \phi(r))$
\end{center}
where $\beta \in ([1,+\infty[\cap \Q)\cup\{+\infty\}$; by convention $\beta = + \infty$ if and 
only if $\cC_1 = \cC_2$, where $\phi$ is a sub-analytic function germ 
tending to $0$ at $0$ and where $\alpha$ is a positive real number.
The \em order of contact between the real analytic half-branches $\cC_1$ and $\cC_2$ \em is the 
rational number $\beta$.  

\smallskip
We are interested in the following properties of blow-analytic homeomorphism germs, which are 
similar to some of (sub-analytic) bi-Lipschitz homeomorphism germs.
\begin{proposition}[\cite{KoPa1}]\label{prop:BA}
Let $h:(\R^2,\oo) \to (\R^2,\oo)$ be a germ of blow-analytic homeomorphism.

\smallskip 
(1) The order of contact between real analytic half-branches is preserved by $h$.  

(2) There exists constants $0<A<B$ such that for $r$ positive and small enough 
\begin{center}
$A|x| \leq |h(x)| \leq B|x|$ once $|x|\leq r$.
\end{center}
\end{proposition}
We will show these properties in Section \ref{section:HP}. 
%
%
%
%
%
%
%
%
%
%
%
%
%
%
%
%
%
%
%
%
%
%
%
%
%
%
%
%
%
%
%
\section{Bi-Lipschitz contact equivalence}\label{section:bLCE}
We give here a short account of the (sub-analytic) bi-Lipschitz contact equivalence with a focus on the special case of plane real analytic function germs. 

We present an equivalent criterion below, to be used to demonstrate the main result of this note, which, although obvious, 
was not observed in \cite{BFGG}.

\medskip
The contact equivalence between (smooth) mappings, as per se, was first introduced by Mather \cite{Mat}.
The natural extension of Mather's definition to the Lipschitz setting in the function case appeared in \cite{BCFR}, 
and to the general case in \cite{RV}
\begin{definition}
Two map-germs  $f,g: (\R^n,0) \longrightarrow (\R^p,0)$ are
called \emph{$\cK$-bi-Lipschitz equivalent} (or
\emph{contact bi-Lipschitz equivalent}) if there exist two  germs of
bi-Lipschitz  homeomorphisms $h:(\R^n,0) \longrightarrow (\R^n,0)$
and  $H: (\R^n \times \R^p,0) \longrightarrow (\R^n \times \R^p,0)$ such
that $H(\R^n \times \{0 \}) = \R^n \times \{0 \}$ and the following
diagram is commutative:

\[
\begin{array}{lllll}
(\R^n,0) & \stackrel{(id,\, f)}{\longrightarrow}
 & (\R^n \times
\R^p,0) & \stackrel{\pi_n}{\longrightarrow} & (\R^n,0) \\
\,\,\,h \, \downarrow &  & \,\,\,\, H \, \downarrow  &  & \,\,\, h \, \downarrow \\
(\R^n,0)& \stackrel{(id, \,g)}{\longrightarrow} & (\R^n \times
\R^p,0) & \stackrel{\pi_n}{\longrightarrow}& (\R^n,0) \\
\end{array}
\]

\medskip

\noindent where  $id:\R^n  \longrightarrow \R^n$ is the identity map
and $\pi_n: \R^n \times \R^p \longrightarrow \R^n$ is the canonical
projection.

The map-germs  $f$ and $g$ are called
\emph{$\mathcal{C}$-bi-Lipschitz equivalent} if  $h=id$.
\end{definition}

We can think of a $\mathcal{C}$-bi-Lipschitz equivalence as being a family of origin preserving bi-Lipschitz maps of $\R^p$ to $\R^p$, parameterized by $\R^n$ which carry $f(x)$ to $g(x)$ for all $x\in\R^n$. If $p=1$, then $\R-\{0\}$ is the union of two contractible sets; each homeomorphism creates a bijection between these components, and this bijection is independent of $x$. This gives an equivalent form for contact equivalence
in the bi-Lipschitz category if $p=1$:

\begin{theorem}[\cite{BCFR}]\label{theorem:CbLE}
Let $f,g:(\R^n,\oo) \to \R$ be two smooth function germs. The functions $f$ and $g$ are
\em bi-Lipschitz contact equivalent \em (or \em $\cK$-bi-Lipschitz equivalent\em ) if there exists a bi-Lipschitz homeomorphism 
germ $h:(\R^n,\oo) \to (\R^n,\oo)$, there exist non zero,  same sign constants $A,B$
and a sign $\sgm \in \{-1,1\}$ such that in a neighborhood of the origin $\oo$ the following 
inequalities hold true 
\begin{center}
$A \cdot g < \sgm \cdot f\circ h < B \cdot g$.
\end{center}
If furthermore $h$ is required to be sub-analytic we will then speak of \em
sub-analytic bi-Lipschitz contact equivalence. \em 
\end{theorem}
\begin{proof} See theorem 2.4 of \cite{BCFR}.
\end{proof}
The inequality enters into the proof of the theorem as follows. The functions $x\to \frac{f(x)}{g(x)}$ 
are bounded by the inequality; this implies that the functions  $(x,y)\to\frac{f(x)}{g(x)}\cdot y$ are Lipschitz on the open set
$0\le |y|\le |g(x)|$, where $y$ is the coordinate on $\R$ (\cite{BCFR}). Since $\frac{f}{g}\cdot y\circ g=f$, the functions 
$\frac{f}{g}\cdot y$ 
on $0\le |y|\le |g(x)|$ can be used as part of a family of bi-Lipschitz homeomorphisms. 
%
%
%
%
%

\medskip
We recall that the bi-Lipschitz contact equivalence of bi-Lipschitz functions has no moduli \cite{BCFR}.  

\medskip
Specializing to the case of $f:(\R^2,\oo) \to \R$, we are going now to recall some material presented 
in \cite{BFGG}. We will keep it to the minimum needed here.

\smallskip
An \em arc \em (\em at the origin\em) is the Puiseux series parameterization $\gm:(\R_{\geq 0},0) \to (\R^2,\oo)$ of a 
given real analytic half-branch such that 
\begin{center}
$|\gm(t)| = t$.
\end{center} 
As an abuse of language we will confuse the notion of arc with its image.

The \em order of contact between two arcs \em means the order of contact
between the respective half-branch images. 

\smallskip
Let $\gm$ be an arc and let $f\in \cO_2$ be a real analytic function germ.
The function germ $t \to f\circ\gm(t)$ is a (converging) Puiseux series 
written as   
\begin{center}
$f\circ\gm(t) = t^\nu(A + \psi(t))$,
\end{center} 
where $\nu \in ]0,+\infty]\cap \Qbar$, with $\Qbar :=\Q\cup\{+\infty\}$; by convention $\nu = + \infty$ if and 
only if $f\circ\gm$ is identically $0$, and otherwise  $\psi$ is a sub-analytic function germ 
tending to $0$ at $0$ and where $A$ is a non-zero real number.
The rational number $\nu_f(\gm) := \nu$ is the \em normalized order \em(\em at the origin\em) 
of the function $f$ along the half-branch $\gm$. 
The \em normalized order \em is the function 
\begin{center}
$\nu_f: \{arcs\; at\; \oo\} \to \Qbar$ defined as $\gm \to \nu_f(\gm)$.
\end{center} 
It has the following properties:
\begin{remark}
The normalized order of $f$ is always larger than or equal to the multiplicity $m_f$ of $f$ at $\oo$.
It is also preserved under sub-analytic bi-Lipschitz contact equivalence.
\end{remark}
A \em H\"older triangle, \em introduced in \cite{Bir}, is any image of the quadrant 
$(\R_{\geq 0}\times\R_{\geq 0},\oo)$ by a sub-analytic homeomorphism  $(\R^2,\oo) \to (\R^2,\oo)$.
The \em exponent \em $\beta(T)$ of a given H\"older triangle $T$ is the order of contact 
at the origin of the boundary curve of $T$ (the images of the axis of the quadrant). 
By extension and also abuse of language if we speak of a H\"older triangle of exponent $+\infty$, we just mean a 
half-branch. A second abuse of language is that any neighborhood of the origin is 
considered as a H\"older triangle with exponent $1$ but boundary-less. 
\begin{remark}
Since (sub-analytic) bi-Lipschitz homeomorphisms preserve the order of contact between
curves, the image of a H\"older triangle by such an homeomorphism is a H\"older triangle
with the same exponent.
\end{remark}
%
%
%
%
\begin{definition}\label{def:width}
Let $\gm$ be an arc, let $f\in \cO_2$ and let $q = \nu_f (\gm) \geq m_f$. 

The \em width \em of $f$ along $\gm$ is the infimum of the exponents taken among 
all the H\"older triangles $T$ containing the arc $\gm$ and such that the normalized order 
$\nu_f$ stays constant and equal to $q$ along any arc contained in $T$. 
This infimum exists and is rational \cite{BFGG}.
It is infinite if and only if $f$ vanishes identically along $\gm$.

We denote this number by $\mu_f(\gm)$.
\end{definition}
\begin{remark}
The width of $f$ along arcs is preserved by the (sub-analytic) bi-Lipschitz contact equivalence.

As can be seen in \cite{BFGG}, the notion of width is at the basis of the construction of the complete invariant 
\em - minimal pizza - \em of a sub-analytic bi-Lipschitz contact equivalence class.
\end{remark}
The criterion we mention is the following result. We give it in the real analytic category, but it is valid in the 
category used in \cite{BFGG} for the same reasons as those of the real analytic category.
\begin{proposition}\label{prop:criterion}
Two real analytic function germs $f,g:(\R^2,\oo) \to (\R,0)$ are (sub-analytically) bi-Lipschitz contact 
equivalent if and only if there exists a (sub-analytic) homeomorphism $h: (\R^2,\oo) \to (\R^2,\oo)$
such that 

(i) for any arc $\gm$, we get 
\begin{center}
$\nu_f(\gm) = \nu_g(h_*\gm)$, 
\end{center}
where $h_*\gm$ is the arc corresponding to the half-branch $h(\gm)$. 

(ii) for any arc $\gm$, we get 
\begin{center}
$\mu_f(\gm) = \mu_g(h_*\gm)$.
\end{center}
\end{proposition}
\begin{proof}
If the functions are already bi-Lipschitz contact equivalent, any equivalence of minimal pizzas 
is such a homeomorphism (see \cite{BFGG}).

\medskip
Assume there exists such a sub-analytic homeomorphism. Conditions $(i)$ and $(ii)$ implies 
that $f$ and $g$ have combinatorially equivalent minimal pizzas \cite[Sections 2\& 3]{BFGG}, which is
sufficient to imply that both functions are sub-analytically bi-Lipschitz contact equivalent. 
 
\medskip 
We do not use that $f,g$ are real analytic, but just that they are sub-analytic  
and Lipschitz in order to apply \cite{BFGG}.
\end{proof}
%
%
%
%
%
%
%
%
%
%
%
%
%
%
%
%
%
%
%
%
%
%
%
%
%
%
%
%
%
%
%
%
%
%
%
\section{Local normal form at exceptional points}\label{section:HP}
This section is devoted to the local form of the blowing down mapping of a finite composition
of point blowings-up initiated with the blowing-up of the origin $\oo$ of $\R^2$.
We can use such semi-local data in order to investigate the asymptotic properties at $\oo$ 
of the restriction of a function germ $f\in \cO_2$ to meaningful H\"older triangles (namely in which the function
is "monomial"). 

\medskip
Let $\pi: (M,E) \to (\R^2,\oo)$ be a finite composition of points blowings-up. 


Local analytic coordinates $(u,v)$ centered at $\ula \in E$ are called \em adapted to $E$ at $\ula$ \em
if the following inclusions for germs hold true
\begin{center}
$\{u=0\} \subset (E,\ula) \subset \{u\cdot v =0\}$.
\end{center}
Any local coordinate system at $\ula \notin E$ is adapted to $E$ at $\ula$.

A (coherent) $\cO_M$-ideal sheaf $\cI$ is principal and monomial in $E$ if it is co-supported in $E$ 
and at any point $\ula$ of $E$, the ideal is locally generated by a monomial in local
coordinates adapted to $E$ at $\ula$:
\\
- if $(u,v)$ are local adapted coordinates at a \em smooth/regular point $\ula$ of $E$, \em that is such that $(E,\ula) =\{u=0\}$, 
we find that $\cI_\ula$ is locally generated by $u^p$ for some non-negative integer $p$;
\\
- if $(u,v)$ are local adapted coordinates at a \em corner point $\ula$ of $E$, \em that is such that 
$(E,\ula) =\{u\cdot v=0\}$, we find that $\cI_\ula$ is locally generated bu $u^pv^q$ for some non-negative 
integers $p,q$. 

\smallskip 
As part of the folklore, we recall the following two facts (both proved by induction on the 
number of blowings-up):

i) The pull back $I_\pi:= \pi^*(\bm_2)$ of the maximal ideal is principal and monomial in $E$  and its
co-support is $E$. 

ii) Let $J_\pi$ be the ideal generated by the determinant of $\ud \pi$.  Then it is also principal and monomial and 
its co-support is also $E$ (see \cite[p. 217]{BM2}). 

\bigskip
The local form of $\pi$ at any point of the exceptional divisor is given in the following
\begin{lemma}\label{lem:res-1}
Let $\pi : (M,E) \to (\R^2,0)$ be the finite composition of point blowings-up presented above.
The mapping $\pi$ has the following properties:

\smallskip
(i) For any non corner point $\ula$ of the normal crossing divisor $E$, there exists $(u,v)$ local 
adapted coordinates to $E$ at $\ula$ such that $(E,\ula) = \{u=0\}$ such that $I_{\pi,\ula} = (u^l)$ and 
$J_{\pi,\ula} = (u^\bl)$ with $\bl\geq 2l-1$. Moreover, (and up to a rotation in the target $(\R^2,0)$)
we write the mapping $\pi$  at $\ula$ as 
\begin{equation}\label{eq:res-map-1}
(u,v) \to (\alpha u^l,u^l \phi(u) + \beta u^mv),
\end{equation} 
where $l\leq m$ are positive integer numbers and the function $\phi$ is analytic and $\alpha,\beta \in \{-1,+1\}$.
\\ 
Note that $\bl=l + m - 1$.
 
\smallskip
(ii) For any non corner point $\ula$ of the normal crossing divisor $E$, there exists $(u,v)$ local 
adapted coordinates to $E$ at $\ula$ such that $(E,\ula) = \{u=0\}$ such that $I_{\pi,\ula} = (u^lv^m)$ and 
$J_{\pi,\ula} = (u^\bl v^\bm)$ with $\bl\geq 2l-1$ and $\bm \geq 2m-1$. Moreover, (and up to a rotation in the target $(\R^2,0)$), locally
the mapping $\pi$  is written at $\ula$ as 
\begin{equation}\label{eq:res-map-2}
(u,v) \to (\alpha u^lv^m,\phi(u^lv^m) + \beta u^nv^p)
\end{equation}
where $l\leq n$ and $m\leq p$ are positive integer numbers
such that the plane vectors $(l,m)$ and $(n,p)$ are not co-linear ($(l,m)\wedge (n,p) \neq 0$), 
the function $\phi:t\to t\cdot h(t^{1/e})$ for a real 
analytic function $h$ and a positive integer $e$, and $\alpha,\beta \in \{-1,+1\}$.
\\
We also see that $\bl=l + n - 1$ and $\bm= m + p- 1$.
\end{lemma}
\begin{proof}
By induction on the number of blowings-up.
The induction step is true when $(M,E) = ((\R^2,\oo),\emptyset)$.
It is thus sufficient to show that properties (i) and (ii) are preserved under point blowings-up.

\smallskip
Let $\bp$ be any point of the exceptional curve $E$. 
Let $\gamma:(M_1,E_1) \to (M,E)$ be the corresponding blowing-down mapping. Let $D:=\gamma^{-1}(\bp)$ 
be the new component of the normal crossing curve $E_1$, and we keep denoting $E$ for 
the strict transform of $E$ by $\gamma$.

\smallskip 
First. Assume that the center $\bp$ is a regular point of $E$.  
Suppose we are given coordinates $(u,v)$ centered at $\bp$ such that point (i) holds true.

\smallskip
In the chart $(x,y) \to (x,xy)$, the exceptional curve $D$ has equation $\{x=0\}$. 
We just write the mapping $\pi\circ\gamma$ as 
\begin{equation}
(x,y) \to (\alpha x^a,x^a \phi(x) + \beta x^cxy) = (\alpha x^a,x^a \phi(x) + \beta x^{c+1}y),
\end{equation}
thus in the form of point (i).

\smallskip
In the chart $(x,y) \to (xy,y)$, the exceptional curve $D$ has equation $\{y=0\}$. 
The mapping $\pi$ is written as 
\begin{equation}
(x,y) \to (\alpha x^ay^a,x^ay^a\phi(xy) + \beta x^cy^{c+1}) 
\end{equation}
At the point $D\cap E$, we find $x=0$.  Thus we obtain a local expression of $\pi$ of type (ii).
\\
At a point $\bp$ of $D\setminus E$  we have $x=A\neq 0$. We assume for simplicity $A=1$
(in order to avoid discussing on the parity of $a$ and $c$). The cases $A$ positive and $A$
negative are dealt with in a very similar way. Let $x := 1+ w$  and $y = (1+w)^{-1}z$. This an isomorphism in a 
neighborhood of the point $\bp$. Note that $(D,\bp) =\{z=0\}$. The mapping $\pi\circ\gamma$ is written as:
\begin{equation}
(w,z) \to (\alpha z^a,z^a \phi(z) + \beta (1+w)^{-1}z^{c+1}).
\end{equation} 
For $|w|<1$, writing $(1+w)^{-1} = 1 + w'$, we see that $(w',z)$ are also coordinates centered at $\bp$. 
We deduce that the mapping $\pi \circ \gamma$ can be written as:
\begin{equation}
(w',z) \to (\alpha z^a,z^a \theta(z) + \beta z^{c+1}w'),
\end{equation} 
where the function $\theta$ is real analytic.
\\
The case of blowing-up a regular point of $E$ has thus been dealt with.

\smallskip
Second. Assume that center $\bp$ of the blowing-up is a corner point of $E$.
We use the same notations as in the non corner case. There exist adapted 
coordinates centered at $\bp$ such that the mapping $\pi$ in coordinates is:
\begin{equation}
(u,v) \to (\alpha u^av^b,\phi(u^av^b) + \beta u^cv^d)
\end{equation}

In the chart $(x,y) \to (x,xy)$ the composed mapping $\pi \circ \gamma$ becomes:
\begin{equation}\label{eq:map-3}
(x,y) \to (\alpha x^{a+b}y^b,\phi(x^{a+b}y^b) + \beta x^{d+d}y^d).
\end{equation}
We observe that the plane vectors $(a+b,b)$ and $(m+n,n)$ are still linearly independent.
\\
At the point $D\cap E$, property (ii) is already satisfied as can be seen in Equation (\ref{eq:map-3}). 
\\
At a point of $D\setminus E$, we know that $y=A \neq 0$. And we proceed as in the second part of the 
first case above.

\smallskip
The proof of the Lemma ends just saying that the blowing-up 
mapping of the origin gives rise to a blowing-down mapping with 
local expression of the form $(x,y) \to (x,xy)$ or $(x,y) \to (xy,y)$, that is of type (i).
\end{proof}
\begin{remark}\label{rmk:HP-exp}
1) The result is clearly true in the complex case as well.

2) The pair of integers $(l,m)$ at a regular point of $E$ is constant along the component of $E$ containing
the point, since they are obtained via $I_\pi$ and $J_\pi$ both principal and monomial in $E$. 

At a corner point the numbers $l,m,n,p$ are just from the pairs $(l,n)$ and $(n,p)$ coming from 
each component through the corner point.
\end{remark} 
In points $(i)$ and $(ii)$ of Lemma \ref{lem:res-1}, the local coordinates $(u,v)$ at $\ula$
are called \em Hsiang \& Pati coordinates at $\ula$ \em (see \cite{HP,Gri1,BBGM,Gra1}).
Following Remark \ref{rmk:HP-exp} we introduce the following
\begin{definition}
Let $H$ be a component of $E$. The pair of integers $(l,m)$ appearing in the local
expression of $\pi$ at (regular) points of $H$ is called the \em Hsiang \& Pati local data of $\pi$ along $H$ \em
(\em Hsiang \& Pati local data \em for short). The union of all the Hsiang \& Pati local data is called 
\em Hsiang \& Pati data of $\pi$ \em (\em Hsiang \& Pati data \em for short).
\end{definition}
Hsiang \& Pati data allows to describe, up to local quasi-isometry (see also \cite{HP,Gri1,BBGM}) 
the local form at any point point of the exceptional locus $E$ of the pull-back of the Euclidean metric 
by the "resolution mapping $\pi$". 

%
%
\begin{corollary}\label{cor:HP}
Let $h:(\R^2,\oo) \to (\R^2,\oo)$ be a cascade blow-analytic homeomorphism
\[
\begin{array}{lll}
(M,E) & \stackrel{\pi'}{\longrightarrow} & (\R^2,\oo) \\
\,\,\,\,\, \Phi \, \downarrow &  & \;\; h \downarrow \\
(M',E') & \stackrel{\pi}{\longrightarrow} & (\R^2,\oo) \\
\end{array}
\]

\smallskip
We find $\Phi^*I_\pi = I_{\pi'}$ and $\Phi^*J_\pi = J_{\pi'}$.

\smallskip
In particular, the Hsiang \& Pati local data 
of $\pi$ along any component $H$ of $E$ and the Hsiang \& Pati local data
of $\pi'$ along $\Phi(H)$ are equal.  
\end{corollary}
\begin{proof}
It is straightforward from the definitions of cascade blow-analytic homeomorphisms and 
of $I_\pi,J_\pi,I_{\pi'}$ and $J_{\pi'}$.
\end{proof}
The next result is a consequence of Corollary \ref{cor:HP} and presents some 
metric properties of germs of blow-analytic homeomorphisms of the plane.
%
%
\begin{proposition}[see also \cite{KoPa1}]\label{prop:BAmetric}
Let $h:(\R^2,\oo) \to (\R^2,\oo)$ be a germ of blow-analytic homeomorphism.

\smallskip\noindent
(1) The order of contact between half-branches is preserved by $h$.  
\\
(2) There exists constants $0<A<B$ such that for $r$ positive and small enough 
\begin{center}
$A|x| \leq |h(x)| \leq B|x|$, once $|x|\leq r$.
\end{center}
\end{proposition}
\begin{proof} The Hsiang \& Pati local data of $\pi$ is used to compute explicitly orders of contact.

\medskip\noindent
$\bullet$ \em Point 1). \em
Let $\gm_1$ and $\gm_2$ be two disjoint arcs at the origin of $\R^2$. 
Let $\Gm_i$ be the strict transform (of the image) of $\gm_i$ by $\pi$.

\smallskip\noindent
{\bf Claim i).} \em Assume that $\Gm_1$ and $\Gm_2$ intersect a same component $H$ of $E$ at two distinct points, $\ula_1$ 
and $\ula_2$. Let $(l,m)$ be the Hsiang \& Pati local data of $\pi$ along $H$. 
The order of contact between $\Gm_1$ and $\Gm_2$ is $\frac{m}{l}$. \em

\smallskip\noindent
{\bf Proof of Claim i).} 
Let $\ulb_1,\ldots,\ulb_s$ be all the corner points of $E$ lying in $H$ indexed in such a way, 
after having chosen an orientation on $H$, that $\ulb_i$ is the successor of $\ulb_{i-1}$ for $i=2,\ldots,s$. 
Let $I_i$ be the "open interval" $]\ulb_i,\ulb_{i+1}[$.

Assume that $\ula_1,\ula_2 \in I_i$ for some $i$. Let $(u,v)$ be local coordinates adapted to $H$.
We can parameterize the half-branches $\Gm_i$ as $c_i:u \to (u,v_i + V_i(u))$ where $V_i$ is 
a Puiseux Series vanishing at $u=0$ and $v_1\neq v_2$.
We find
\begin{center}
$
\begin{array}{rcl}
\pi(c_i(u)) & = & (\alpha u^l, u^l\phi(u) +\beta u^m(v_i+V_i(u)) \\
&  = & (\alpha t,tg(t) +\beta t^\frac{m}{l}(v_i + W_i(t))
\end{array}
$
\end{center}
where $W_i$ is a Puiseux series vanishing at $t=0$. Thus the order of contact between $\Gm_1$ and $\Gm_2$ is $\frac{l}{m}$.

\smallskip
Assume now that $\ula_1 = \ulb_i$ and $\ula_2, \in I_{i-1}\cup I_i$.
Let $(u,v)$ be local coordinates at $\ulb_i$ such that $\ula_2$ also lies in the domain of the chart.
Let $c_i:u\to (u,v_i + V_i(u))$ be a local parameterization of $\Gm_1$ and $\Gm_2$ where $V_i$ is a 
Puiseux series vanishing at $u=0$ and $v_1 = 0 \neq v_2$.
We see 
\begin{center}
$
\begin{array}{rcl}
\pi(c_i(u)) & = &(\alpha u^l(v_i + V_i)^b,u^l(v_i+V_i)^m g(u,v) + \beta u^n (v_i + V_i)^p) \\
& = & (\alpha t,t g(t) + \beta t^\frac{n}{l} (v_i + V_i)^{p-\frac{mn}{l}}).
\end{array}
$
\end{center}
Thus the order of contact is again $\frac{m}{l}$.

\smallskip
If $\ula_1 \in [\ulb_i,\ulb_{i+1}]$ and $\ula_2\in [\ula_j,\ula_{j+1}]$ with $|i-j|\geq 2$,
the proof is deduced from the cases above using intermediary real analytic arcs $C_k$ with $C_k \cap H = \ula_k \in I_k$ for 
$i<k<j$ (or $j<k<i$) to get the order of contact between $\Gm_1$ and $\Gm_2$.

\smallskip\noindent
{\bf Claim ii).} \em 
Assume that $\Gm_1 \cap E = \{\ula_1\} \in H_1$ and $\Gm_2\cap E = \{\ula_2\}\in H_2$ 
such that $\ula_1 \neq \ula_2$ and $H_1 \cap H_2 =\{\ula\}$. Let $(l,n)$ and $(m,p)$ 
be the Hsiang \& Pati local data of $\pi$ along $H_1$ and respectively along $H_2$.
The order of contact of $\Gm_1$ and $\Gm_2$ is $\min(\frac{n}{l},\frac{p}{l})$. \em

\smallskip\noindent
{\bf Proof of Claim ii).} We can assume that $\Gm_1$ and $\Gm_2$ are contained in a chart centered at the corner point 
$\ula$. Let $C$ be a real analytic half-branch at the origin $\oo$ of $\R^2$ such that its strict transform $C'$ by $\pi$ 
intersects with $E$ at $\ula$.
The contact between $\Gm_1$ and $\Gm_2$ is the minimum of the orders of contact between $\Gm_1$ and $C'$ and 
between $\Gm_2$ and $C'$. In other words, using part i)
it is $\min(\frac{n}{l},\frac{p}{n})$. 

\smallskip\noindent
{\bf Claim iii).} \em Suppose $\Gm_1\cap E = \Gm_2\cap E = \{\ula\}$.
\\
- If $\ula$ is not a corner point there are local coordinates $(u,v)$ at $\ula$ 
such that $(E,\ula) = \{u=0\}$. 
Let $c_i :u\to (u,u^{e_i}A_i)$ be parameterizations of the strict transforms 
of $\Gm_i$ with $e_i \in \Q_{>0}$ and $A_i$ is an invertible Puiseux series if not identically $0$, for $i=1,2$. 
Let $u^{e_1}A_1 - u^{e_2}A_2 = u^e A$ for an invertible Puiseux unit $A$.
The order of contact between $\Gm_1$ and $\Gm_2$ is $\frac{m+e}{l}$.
\\
- If $\ula$ is a corner point, there are local coordinates $(u,v)$ such that $(E,\ula) = \{uv=0\}$.
Let $c_i :u\to (u,u^{e_i}A_i)$ be parameterizations of the strict transforms 
of $\Gm_i$ with $e_i \in \Q_{>0}$ and $A_i$ is an invertible Puiseux series if not identically $0$, for $i=1,2$. 
Let $u^{e_1}A_1 - u^{e_2}A_2 = u^e A$ for an invertible Puiseux unit $A$.
The order of contact between $\Gm_1$ and $\Gm_2$ is 
\begin{center}
$\frac{n+p\cdot\min(e_1,e_2) + e - \min(e_1,e_2)}{l+m\cdot\min(e_1,e_2)}$.
\end{center}
\em

\smallskip\noindent
{\bf Proof of Claim iii).} 
Assume $\ula$ is not a corner point. 
Let $u^{e_1}A_1 - u^{e_2}A_2 = u^e A$ with $e\in \Q_{>0}$ and $A$ is an invertible Puiseux series.
We have $\pi(c_i(u)) = (\alpha u^l,u^lg(u) +\beta u^l\cdot u^{e_i} A_i (u)) = (\alpha t,tg(t) +\beta t^\frac{m+e}{l} B(t))$ 
where $B$ is an invertible Puiseux series. Thus the contact is $\frac{m+e}{l}$.

\smallskip\noindent
Assume that $\ula$ is a corner point. Assume that $e_1\leq e_2$.
Then 
\begin{center}
$
\begin{array}{rcl}
\pi(c_1(u)) & = &(\alpha u^{l+me_1}M^m,u^{l+me_1}M^m A(u) +\beta u^{n+pe_1}M^p) \\
& = & (\alpha t,tg(t) + \beta t^\frac{n+pe_1}{l+me_1}M^{p-\frac{mn}{l}}) \\
\pi(c_2(u)) & = & (\alpha u^{l+me_1}N^m,u^{l+me_2}N^m(u)B(u) + \beta u^{n+pe_1}N^p) \\
& = & (\alpha t,tg(t) + \beta t^\frac{n+pe_2}{l+me_2}N^{p-\frac{mn}{l}}).
\end{array}
$
\end{center}
If $e = e_1$ then the contact is $\frac{n+pe}{l+me}$
\\
If $e > e_1=e_2$ then 
\begin{center}
$u^{e_1}M - u^{e_1}N = u^e P = t^\frac{e_1}{l+me_1} [M(t(u)) - N(t(u)] = t^\frac{e}{l+me_1} P(t(u))$
\end{center}
so that the contact is $\frac{n+pe_1 + e - e_1}{l+me_1}$.  

\smallskip\noindent
{\bf Claim iv).} \em Suppose that $\Gm_1 \cap E = \ula_1 \in H_1$ and $\Gm_2\cap E = \ula_2 \in H_2$
with $H_1\cap H_2 = \emptyset$. There is a unique chain $D_0,\ldots,D_s$ of distinct components 
of $E$ such that $D_0:= H_1$, $D_s:= H_2$ and $D_i\cap D_j = \emptyset$ if $0 \geq i \geq j-2 \geq s$
and $D_i\cap D_{i+1}$ consists in a single corner point. Let $(l_i,m_i)$ be the Hsiang \& Pati local data of $\pi$ along 
$D_i$. The contact between $\Gm_1$ and $\Gm_2$ is 
\begin{center}
$\min_{i=0,\ldots,s}\frac{m_i}{l_i}$. 
\end{center}
\em 

\smallskip\noindent
{\bf Proof of Claim iv).} This is similar to the case ii).

\smallskip\noindent
Following Corollary \ref{cor:HP} and the expressions of the order of contact given above,
the order of contact between two half branches $C_1$ and $C_2$ is 
equal to the order of contact of the half-branches $h(C_1)$ and $h(C_2)$. 

\medskip\noindent
$\bullet$ \em Point 2). \em
Let $C$ be a half-branch at the origin $\oo$. 
Let $\ula \in E$ be the intersection point of the strict transform $C'$ of $C$ by $\pi$ 
with $E$. Let $(u,v)$ be local coordinates at $\ula$ adapted to $E$ so that 
$\{u=0\} \subset(E,\ula) \subset \{u\cdot v=0\}$.

If $\ula$ is a regular point of $E$, let $(l,m)$ be the Hsiang \& Pati local data of  
$\pi$ at $\ula$, so that $|\pi(u,v)| = |u|^l\cdot A(u,v)$ with $A(0,0) >0$.
  
If it is a corner point, the Hsiang \& Pati local data of $\pi$ at $\ula$ consists of two pairs $(l,n)$ 
and $(m,p)$ and we get $|\pi(u,v)| = |u|^l|v|^m\cdot A(u,v)$ with $A(0,0) >0$.

\smallskip
Suppose $\ula$ is a corner point. Up to permuting $u$ and $v$, let $c:t \to (\pm t,t^r\phi(t)$, with $t\in \R_{\geq 0}$, 
be a Puiseux parameterization of $C'$ such that $\phi$ is an invertible Puiseux series
and $r\in \Qbar_{\geq 1}$. Thus 
\begin{center}
$|\pi(c(t))| = t^{l+mr}A(t)$ for a Puiseux series $A$ with $A(0) >0$,
\end{center}
and 
\begin{center}
$|\pi'(\Phi(c(t)))| = t^{l+mr}B(t)$ for a Puiseux series $B$ with $B(0) >0$.
\end{center}
Similar but simpler computations occur at a non corner point.

\smallskip
The above computation says that $L|x| \leq |h(x)| \leq  K|x|$ in the image of the considered coordinates 
chart $(u,v)$, which is a finite union of H\"older triangles and $0<L<K$.   
By compactness, the exceptional curve $E$ is covered by finitely many such charts. Positive constants 
$K,L$ can be found such that such an inequality of the desired type holds true in a neighborhood of the origin.
\end{proof}
\begin{remark}
Point 1) implies than any blow-analytic image of a H\"older triangle is a H\"older triangle
with equal exponent.
\end{remark}
%
%
%
%
%
%
%
%
%
%
%
%
%
%
%
%
%
%
%
%
%
%
%
%
%
%
%
%
%
%
%
%
%
\section{Proof of the main result}\label{section:main}
This section is devoted to proof the main result of this note, namely 
\begin{theorem}\label{thm:main}
Blow analytic equivalent real analytic plane function germs are sub-analytically bi-Lipschitz
contact equivalent.
\end{theorem}
Let $f,g \in \cO_2$ two blow-analytic equivalent function germs. 
What we exactly show in this section is the following result which, when combined 
with the results of \cite{BFGG} about \em abstract pizzas, \em will yield Theorem \ref{thm:main}.
\begin{proposition}\label{prop:BA-pizza}
Let $h$ be the blow analytic isomorphism between $f$ and $g$, say $f=g\circ h$. 
Thus the homeomorphism $h$ maps any pizza of $f$ onto an equivalent pizza of $g$ (although
$h$ may never realize any equivalence between any two pizzas).
\end{proposition}

A non constant function germ $f\in \cO_2$ is said \em resolved \em or \em monomialized \em 
if there exists a finite composition of points blowings-up $\pi:(M,E)\to (\R^2,\bo)$
such that (see \cite{Hir,BM1,BM2})
\\
- the co-support of the pull-back $I_f:=\pi^*((f))\subset I_\pi$ 
of the ideal $(f)$ is a normal crossing divisor $E\cup F$, where $F$ is the strict transform of 
$f^{-1}(0)$; 
\\
- the ideal $I_f$ is principal and monomial in $E\cup F$.

\smallskip
Let $f \in \cO_2$ and let $\pi_f:(M_f,\cE_f,E_f) \to (M,f^{-1}(0),\oo)$ be a (minimal) resolution
of $f$. Since the ideal $I_f = \pi_f^*((f))$ is principal and monomial in the normal crossing divisor 
$\cE_f$. Let $\ula\in E_f$ and let $(u,v)$ be local coordinates at $\ula$ adapted to $E_f$ and $\cE_f$, that 
is 
\begin{center}
$\{u=0\}\subset (E_f,\ula) \subset (\cE_f,\ula) \subset \{u\cdot v= 0\}$.
\end{center}

If $\ula$ is a regular point of $\cE_f$, we get $I_{f,\ula} = (u^r)$ for $r\in\N_{\geq m_f}$. 

If $\ula$ is a corner point of $\cE_f$, we have $I_{f,\ula} = (u^rv^s)$ for $r,s\in\N_{\geq m_f}$. 

\smallskip
The number $r$ depends only on the component of $E_f$ through the point $\ula$, while 
the number $s$ depends only on the component of $\cE_f$ through $\ula$.  
%
%
%
\begin{definition}
The multiplicity(ies) $r$ ($r,s$) is (are) called \em the local resolution data of $f$ at $\ula\in E_f$. \em
The (finite) collection of all these integer numbers is called \em the resolution data of $f$. \em 
\end{definition}
Let $f,g \in \cO_2$ be two blow-analytic equivalent function germs. Let $h$ be the corresponding 
blow-analytic homeomorphism obtained from the following cascade blow-analytic commutative diagram:
\[
\begin{array}{lll}
(M_f,\cE_f:=\pi_f^{-1}(F),E_f) & \stackrel{\pi_f}{\longrightarrow} & (\R^2,F:=f^{-1}(0),\oo) \\
 \;\;\;\;\;\;  \;\;\;\;\;\;  \;\;\;\; \Phi \, \downarrow &  & \;\;\;\;\;\;  \; \;\;\;\;\;\; h \downarrow \\
(M_g,\cE_g:=\pi_g^{-1}(G),E_g) & \stackrel{\pi_g}{\longrightarrow} & (\R^2,G:=g^{-1}(0),\oo) \\
\end{array}
\]
where $\cE_f := \{f\circ\pi_f= 0\}$ and $\cE_g:=\{g\circ\pi_g=0\}$.
Having $f=g\circ h$ is equivalent to have $f\circ \pi_f = g\circ\pi_g\circ\Phi$, so that 
$\Phi^*(\pi_g^*((g)))= \pi_f^*((f))$.

Any point blowing-up $\tau:(\widetilde{M},\wtE) \to (M_f,E_f)$ with center $\bp$ in $M_f$ 
leads to a point blowing-up $\tau':(\widetilde{M}',\wtE') \to (M_g,E_g)$  with center $\Phi(\bp)$, 
so that we can extend $\tau^*\Phi$ as an analytic isomorphism $(\widetilde{M},\wtE) \to(\widetilde{M}',\wtE')$.
This allows to further assume that $\pi_f$ is a (minimal) resolution of $f$ which (equivalently) implies that
$\pi_g$ is also a (minimal) resolution of $g$.
Under this additional property $\cE_f$ and $\cE_g$ are both normal crossing divisor.
 
\medskip
The proof of the main result starts with the following
\begin{lemma}\label{lem:order-arc}
Let $\gm$ be a real analytic arc along which $f$ does not vanish identically. 
Let $\Gm$ be the strict transform of (the image of) $\gm$ which intersects with $E_f$ at $\ula$. 
Let $(u,v)$ be local coordinates at $\ula$ adapted to $E_f$ and $\cE_f$.

\smallskip 
The normalized order $\nu_f(\gm)$ of $f$ along $\gm$ is a fraction whose numerator is affine with coefficients 
in the local resolution data of $f$ at $\ula$, the denominator is affine functions with coefficients in the Hsiang \& Pati 
local data of $\pi_f$ at $\ula$, and the variable is $p$ as in the parameterization $t \to (\pm t, t^p \theta(t))$ 
of $\Gm$, with $p \in \Q_{>0}$ and where $\theta$ is an invertible Puiseux series. 
\end{lemma}
\begin{proof}
We find below in all possible situations the explicit expression of $\nu_f(\gm)$.
Let $\gm_f:t \to (\pm t, t^p \theta(u))$  be a Puiseux parameterization of $(\Gm,\ula)$.

\medskip\noindent
$\bullet$ \em Assume that $\{\ula\} = \Gm\cap E_f$ is a regular point of $E_f$ with 
$\ula \notin \clos(\cE_f\setminus E_f)$. Since $I_{\pi_f,\ula} = (u^l)$ and $I_{f,\ula} = (u^r)$,
the normalized order of $f$ along $\gm$ is  
\begin{center}
$\nu_f(\gm) = \frac{r}{l}$
\end{center}
\em 
In this case, we find that 
\begin{center}
$f\circ \gm_f(t) = t^r A(t)$
\end{center} 
for an invertible Puiseux series $A$. 
Since $|\pi_f\circ\gm_f(t)| = |t|^l B(t)$ for a positive Puiseux series $B$, we get the results.

\medskip\noindent
$\bullet$ \em Assume that $\ula$ is a corner point of $\cE_f$ but a regular point of $E_f$. 
Thus $I_{\pi_f,\ula} = (u^l)$ and $I_{f,\ula} = (u^rv^s)$. 
The normalized order of $f$ along $\gm$ is  
\begin{center}
$\nu_f(\gm) = \frac{r+sp}{l}$.
\end{center}
\em 
We find
\begin{center}
$f\circ \gm_f(t) = t^{r+ps} A(t)$ and $|\pi_f\circ\gm_f(u)| = |t|^l B(t)$ 
\end{center}
for invertible Puiseux series $A,B$. So we deduce the
normalized order.  

\medskip\noindent
$\bullet$ \em Assume that $\{\ula\} = \Gm\cap E_f$ is a corner point of $E_f$,
so that $I_{\pi_f,\ula} = (u^lv^m)$ and 
$I_{f,\ula} = (u^rv^s)$. 
The normalized order of $f$ along $\gm$ is  
\begin{center}
$\nu_f(\gm) = \frac{r+sp}{l+mp}  = \frac{rq+s}{lq+m}$.
\end{center}
\em
This last situation yields
\begin{center}
$f\circ \gm_f(t) = t^{r+ps} A(t)$ and $|\pi_f\circ\gm_f(t)| = |t|^{l+pm} B(t)$
\end{center}
for invertible Puiseux series $A,B$. 
The case of a parameterization of the form $\tau\to (\tau^q\mu (\tau),\pm \tau)$ 
works similarly.
\end{proof}
A key to our proof consequence of Lemma \ref{lem:order-arc} is the following 
\begin{proposition}\label{prop:main}
Let $\gm$ be an arc and let $h_*\gm$ be the arc corresponding to the image $h\circ\gm$.
We find 
\begin{equation}\label{eq:main-1}
\nu_f(\gm) = \nu_g (h_*\gm)
\end{equation}
\end{proposition}
\begin{proof}
The function $f$ vanishes identically along $\gm$ if and only if $g$ vanishes identically along $h_*\gm$.

\smallskip
Assume that $\gm$ is not contained in $f^{-1}(0)$, thus $h_*\gm$ is not contained in
$g^{-1}(0)$ either.

Following Section \ref{section:HP}, such as Corollary \ref{cor:HP}, 
all the integers linked to $\pi_f$, $f$ and $\Gm$ used to compute the normalized order of $f$ along the arc $\gm$  
in Lemma \ref{lem:order-arc} are inherited by $\pi_g$, $g$ and $\Phi(\Gm)$ via 
the analytic isomorphism $\Phi$.
\end{proof}
%
%
Combining the fact that, any blow analytic image of a H\"older triangle 
is a H\"older triangle of equal exponent, with Proposition \ref{prop:main} the proof of our main 
result is, in principle,
already "done". 
We are now ready to go into the
\begin{proof}[Proof of Proposition \ref{prop:BA-pizza}]
From Lemma \ref{lem:order-arc} and Proposition \ref{prop:main} we deduce that the equivalent 
criterion of 
sub-analytic bi-Lipschitz contact equivalence presented as Proposition \ref{prop:criterion}
is satisfied, concluding the proof.
\end{proof}
%
%
%
%
%
%
%
%
%
%
%
%
%
%
%
%
%
%
%
%
%
%
%
%
%
%
%
%
%
%
%
%
%
%
\section{Appendix: a sketch of the existence of Pizzas}\label{section:app}

We present here a sketch of the existence of pizza (in the case of a plane real analytic 
function germ).

The paper \cite{BFGG} uses a general Preparation Theorem for the functions considered there (more general
than continuous sub-analytic plane function germs), which, roughly speaking, states that the function behaves
like a monomial in some H\"older-like triangle, so that we can find finitely many of them to cover a neighborhood
of the considered point. 
In the case of plane real analytic function germs, and in relation with blow-analytic equivalence
we are going to use here the resolution of singularities of the given function in order to provide explicit,
in the combinatorial data (function and resolution mapping), expressions for the width.  

\medskip
Let $f\in \cO_2$ be given and let $\pi : (M_f,\cE_f,E_f) \to (E,f^{-1}(0),\oo)$ be the (minimal) resolution
of $f$.

Let $\ula$ be a point of $E_f$ and let $(u,v)$ local coordinates at $\ula$ adapted to 
$E_f$ and $\cE_f$ so that 
\begin{center}
$\{u=0\} \subset (E_f,\ula) \subset (\cE_f,\ula) \subset \{u\cdot v = 0\}$.  
\end{center}
We also know that 
\begin{center}
$I_{\pi,\ula}  = (u^lv^m)$ and $J_{\pi,\ula}  = (u^\bl v^\bm)$ and $I_{f,\ula}  = (u^rv^s)$.
\end{center} 
With certainty we know that $\bl \geq 2l -1 \geq 1$, and $\bm\geq 2m -1 \geq 1$  and $r \geq m_f \cdot l$. 
We can assume that we are working in a (semi-analytic) box $B_\ula := ]-u_\ula,u_\ula[ \times ]-v_\ula,v_\ula[$
for positive real number $u_\ula$ and $v_\ula$.

The image $\pi(\clos(B_\ula))$ is a finite union of (one, two, or four) H\"older triangle(s) 
(with equal exponent) which can be calculated by the Hsiang \& Pati local data 
of $\pi$ at $\ula$.

We can check that $f$ is monotonic on each such H\"older triangle, and with computations 
similar (but longer) to those done in the proofs of Proposition \ref{prop:BAmetric} and of 
Lemma \ref{lem:order-arc} we can explicitly find the expression of the (non directed) 
width $\omg_f^T$ in terms of the exponents $l,m,\bl,\bm,r,s$.

Since $E_f$ is compact we can cover  it with finitely many closure of (semi-analytic) boxes 
of the type $B_\ula$ above in such a way the intersection of the closures of two
such boxes has always empty interior. 
This partition provides the triangulation of a neighborhood of the origin by finitely 
many H\"older triangles (the images by the resolution mapping $\pi$ of the closures of all the boxes) thus and 
pizza adding the (directed) width and the signs of the function within the interior 
of these \"older triangles. 

\bigskip
Let us make a final comment:
\begin{remark}
Either way (Preparation Theorem or minimal resolution) picked to cook-up a pizza of a plane real 
analytic function germ,  the combinatorial data to handle to describe the width functions
may be huge. Thus, when the function has a simpler combinatorial data (such as a non-degenerate Newton diagram) 
we are hopeful to describe, by means of this simpler combinatorial data, pizzas attached to $f$ \cite{BFGafG}. 
\end{remark}
%
%
%
%
%
%
%
%
%
%
%
%
%
%
%
%
%
%
%
%
%
%
%
%
%
%
%
%
%
%
%
%
%
%
%
%
%
%
%
%
%
%
%
%
%
%
%
%
%
%
%
%
%
%
%
%

%
%
%
%
%
%
%
%
%
%
%
%
%
%
%
%
%
%
%
%
%
%
%
%
%
%
%
%
%
%
%
%
%
%
%
%
%
%
%
%
%
%
%
%
%
%
%
%
%
%
%
%
%
%
%

\end{document}